\documentclass[11pt]{amsart}
\usepackage{amsmath,amssymb}
\usepackage{graphicx}
\allowdisplaybreaks

\usepackage{bm}
\usepackage{subfig,tikz,stmaryrd}
\usepackage{stmaryrd}
\usepackage{wrapfig}
\usepackage{xspace}
\usepackage{enumerate}
\usepackage{verbatim}
\theoremstyle{plain}
\newtheorem{theorem}{Theorem}[section]
\newtheorem{lemma}[theorem]{Lemma}
\newtheorem{corollary}[theorem]{Corollary}

\newtheorem{proposition}[theorem]{Proposition}
\theoremstyle{definition}

\numberwithin{equation}{section}

\newcommand{\tr}{\trace_{h_x}}
\newcommand{\tf}{\tracefree_{h_x}}

\DeclareMathOperator{\trace}{tr}

\DeclareMathOperator{\ric}{Rc}
\DeclareMathOperator{\scal}{R}
\DeclareMathOperator{\tracefree}{tf}
\DeclareMathOperator{\hess}{Hess}
\DeclareMathOperator{\divergence}{div}
\DeclareMathOperator{\LWT}{LWT}

\DeclareMathOperator{\LOT}{LOT}

\begin{document}

\title[ ]{Formal power series for asymptotically hyperbolic Bach-flat metrics}

\author{Aghil Alaee}
\address{\parbox{\linewidth}{Aghil Alaee\\
		Department of Mathematics and Computer Science, Clark University, Worcester, MA 01610, USA\\
		Center of Mathematical Sciences and Applications, Harvard University, Cambridge, MA 02138, USA}}
\email{aalaeekhangha@clarku.edu, aghil.alaee@cmsa.fas.harvard.edu}

\author{Eric Woolgar}
\address{\parbox{\linewidth}{Eric Woolgar\\
Department of Mathematical and Statistical Sciences, University of Alberta, Edmonton AB, Canada T6G 2G1}}
\email{ewoolgar@ualberta.ca}

\begin{abstract} It has been observed by Maldacena that one can extract asymptotically anti-de Sitter Einstein $4$-metrics from Bach-flat spacetimes by imposing simple principles and data choices. We cast this problem in a conformally compact Riemannian setting. Following an approach pioneered by Fefferman and Graham for the Einstein equation, we find formal power series for conformally compactifiable, asymptotically hyperbolic Bach-flat 4-metrics expanded about conformal infinity. We also consider Bach-flat metrics in the special case of constant scalar curvature and in the special case of constant $Q$-curvature. This allows us to determine the free data at conformal infinity, and to select those choices that lead to Einstein metrics. The asymptotically hyperbolic mass is part of that free data, in contrast to the pure Einstein case. Higher dimensional generalizations of the Bach tensor lack some of the geometrical meaning of the 4-dimensional case, but for a generalized Bach equation suited to the Fefferman-Graham technique, we are able to obtain a relatively complete result illustrating an interesting splitting of the free data into low-order ``Dirichlet'' and high-order ``Neumann'' pairs.
\end{abstract}

\maketitle
\section{Introduction}\label{section1}
\setcounter{equation}{0}

\noindent In seminal work, Fefferman and Graham \cite{FG1, FG2} studied formal series solutions of the Einstein equation for asymptotically hyperbolic metrics expanded about conformal infinity. This led to the identification of data for the singular boundary value problem for these metrics, the discovery of obstructions to power series solutions, and ultimately the discovery of new conformal invariants for the conformal boundary. It also laid the groundwork for holography within the AdS/CFT correspondence.

More recently, Gover and Waldron \cite{GW} and Graham \cite{Graham1} have performed similar analyses for a scalar geometric PDE problem, a singular boundary value problem for the Yamabe equation. In 3-dimensions, this problem was solved in \cite{ACF} as part of the construction of hyperboloidal initial data for the Einstein equations on spacetime. Albin \cite{Albin} has announced an analysis of asymptotically hyperbolic formal series solutions of the Euler-Lagrange equations of Lovelock actions in arbitrary dimensions.

Here we study the question of formal series expansions for a fourth-order geometric PDE in the asymptotically hyperbolic setting. We focus on the Bach equation in dimension $n=4$. The Bach equation is
\begin{equation}
\label{eq1.1}
0= B_{ac}:= \frac{1}{(n-3)}\nabla^b\nabla^d W_{abcd}+\frac{1}{(n-2)}W_{abcd}R^{bd}\ ,\\
\end{equation}
where $W_{abcd}$ is the Weyl tensor, $R_{ab}$ is the Ricci tensor, and $B_{ab}$ is called the Bach tensor. On closed $4$-manifolds, \eqref{eq1.1} is the Euler-Lagrange equation for the functional
\begin{equation}
\label{eq1.2}
\mathcal{W}(g)= \int_M \left \vert W_g\right \vert^2dV_g\ ,
\end{equation}
though for $n\ge 5$, critical points of this functional are all locally conformally flat (as the trace part of the Euler-Lagrange functional is $\frac12(n-4)|W|^2$ and vanishes iff $W=0$) and therefore satisfy \eqref{eq1.1} somewhat trivially. There are inequivalent ways to extend the Bach tensor, originally defined only for $n=4$, to higher dimensions (see sections 1.3 and 5, and reference \cite{Bergman}).

We seek solutions of \eqref{eq1.1} with a pole at infinity of order 2, expressible as
\begin{equation}
\label{eq1.3}
g=\frac{1}{x^2}\left ( dx^2 \oplus h_x \right )\ ,
\end{equation}
on a complete $4$-manifold $(M,g)$, or at least on a collar neighbourhood $x< \epsilon$, where $h_x$ extends differentiably to $x=0$ and induces a Riemannian metric on each constant-$x$ hypersurface. Metrics obeying \eqref{eq1.3} are called \emph{conformally compactifiable and asymptotically hyperbolic}.
Metrics expressed in the form of \eqref{eq1.3} are said to be in \emph{(Graham-Lee) normal form}.
For small $x$, an open region in $(M,dx^2\oplus h_x)$ may be isometrically embedded as an open, bounded region in a product manifold in which the locus $x=0$ becomes a boundary. This locus, equipped with the conformal class $\left [h_0 \right ] $, is called \emph{conformal infinity}, for which $x$ is called a \emph{special defining function} (see Section 2 for more terminology).\footnote
{Throughout we use coordinates $x^a=(x,y^i)$ on a collar neighbourhood of conformal infinity, where $x$ is a special defining function for conformal infinity and the $y^i$ are coordinates in the level sets $x=const$ orthogonal to $\partial_x$. Greek indices are generally reserved for other purposes.}
When such a metric is Einstein, it is called \emph{Poincar\'e-Einstein}. We will use the term \emph{Poincar\'e-Bach} for Bach-flat metrics of the form \eqref{eq1.3}. As with the Poincar\'e-Einstein case \cite{FG1, FG2}, we will pursue here the goal of finding formal power series for $h_x$ for Poincar\'e-Bach metrics. We do not consider convergence, not even on a collar of $x=0$. The question of asymptotics, and in particular the existence of smooth conformal compactifications, for asymptotically hyperbolic solutions of the $n=4$ Bach equation on a neighbourhood of conformal infinity was taken up by Anderson in \cite{Anderson2, Anderson3}.

Define
\begin{equation}
\label{eq1.4}
\begin{split}
E_g:=&\, \ric_g+(n-1)g\ ,\\
A_g:=&\, \trace_g E_g = \scal_g+n(n-1)\ ,
\end{split}
\end{equation}
(we sometimes omit the subscript $g$) and recall that a conformally compactifiable and asymptotically hyperbolic metric has $E_g={\mathcal O}(x)$. If $E_g={\mathcal O}(x^2)$, a calculation shows that
\begin{equation}
\label{eq1.5}
h_x'(0)=0\ ,
\end{equation}
where we denote differentiation with respect to $x$ by a prime. We recall (following terminology in \cite{DGH}) that a conformally compactifiable metric is \emph{asymptotically hyperbolic Einstein to order $k$} if $E_g\in {\mathcal O}(x^k)$ for $x$ any special defining function; this is also called \emph{asymptotically Poincar\'e-Einstein (APE) to order $k$}. Any metric that is APE to order $2k<n-1$ is \emph{partially even to order $2k$}, by which we mean that the odd-order derivatives $h^{(2j-1)}(0)$ vanish for $j\le k$.

\subsection{Four bulk dimensions} A major motivation for the present paper is the assertion of Maldacena \cite{Maldacena} that in $n=4$ bulk dimensions one can replace Einstein gravity by classical conformal gravity, the variational theory of the action functional \eqref{eq1.2} with suitable asymptotically anti-de Sitter or asymptotically hyperbolic fall-off conditions and other conditions. Maldacena's proposal is that the condition $h_x'(0)=0$, together with certain physical considerations, selects only those critical points of this action which are Einstein. For another approach, based on Anderson's formula \cite{Anderson1} for renormalized volume but ultimately invoking other considerations as well, see \cite{AO}.

In the spirit of holography, we instead search for well-defined asymptotic conditions which alone can select Einstein metrics, at least when considering Riemannian signature metrics. This brings us to our main result.

\begin{theorem}\label{theorem1.1} Let $h_0$ be a Riemannian metric on $\Sigma^3$ and let $\Phi$, $\Psi$ be smooth symmetric $h_0$-tracefree $(0,2)$-tensors on $\Sigma$ such that $\divergence_{h_0}\Psi=0$. Let $T_i$ denote smooth functions on $\Sigma$ for $i\ge 2$. For any such data $h_0$, $\Phi$, $\Psi$, $T_i$ with $i\ge 2$, the equation $B_g=0$ admits a unique normal form solution \eqref{eq1.3} on $(M^4,g)$ with $h_x\equiv h(x)$ given by a formal power series in $x$, such that $(\Sigma,[h_0])=\partial_{\infty}M$ is the conformal infinity, with $h(0)=h_0$, $h'(0)=0$, $\tracefree_{h_0} h''(0)=\Phi$, $\tracefree_{h_0} h'''(0)=\Psi$, and $\trace_{h_0}h^{(i)}(0)=T_i$.
\end{theorem}

Here $\tracefree_{h_0}$ and $\trace_{h_0}$ denote the tracefree and trace parts of a $(0,2)$-tensor (converted to an endomorphism using $h_0$), respectively.
The data encoded in $T_i$ are an artifact of conformal freedom (they can be made to vanish by choosing an appropriate conformal representative; see the next subsection).

In \cite{Anderson3}, the problem of boundary data for Bach-flat metrics in 4-dimensions is studied in harmonic gauge rather than in the Graham-Lee normal form gauge of equation \eqref{eq1.3}, with analogous results to ours.

In \cite{GILM}, the variations of the on-shell action \eqref{eq1.2} (i.e., the value of the action at a Bach-flat metric) in $4$-dimensions due to variations in $h_0:= h(0)$ and $h'(0)$ are computed. The variation with respect to $h_0$ is complicated but reduces when $h'(0)=0$ to (a constant times) the third-order (in powers of $x$) piece of the electric components $W(\cdot,\partial_x,\cdot,\partial_x)$ of the Weyl tensor evaluated at conformal infinity \cite[Equation (29)]{GILM}. A simple calculation then yields that this variation is just proportional to $\Psi$. The variation with respect to $h'(0)$ is (a constant times) the second-order piece of the electric Weyl tensor at infinity \cite[Equation (24)]{GILM},
which we compute to be $-\frac12(\Phi+\tracefree_{h_0} \ric_{h_0})$, or simply $-\frac12 \Phi$ when $h_0$ is Einstein.

\subsection{Choosing the conformal representative and the mass aspect}
In view of \cite{Maldacena}, one can try to find the subset of formal power series for Bach-flat metrics which are formal power series for Poincar\'e-Einstein $4$-metrics. Such metrics have $h'(0)=0$ and $h''(0)=-2P_{h_0}$, where $P_{h_0}$ denotes the Schouten tensor of $h_0$. However, the $4$-dimensional Bach tensor is conformally invariant. Its vanishing is an integrability condition for \emph{conformally} Einstein metrics. To choose Einstein representatives within conformal classes of metrics, one must impose a further condition that will fix the trace data in Theorem \ref{theorem1.1}. Now, Einstein metrics obviously have constant scalar curvature $\scal_g=-12$ and constant Branson $Q$-curvature $Q_g=6$ where
\begin{equation}
\label{eq1.6}
Q_g:= \frac16 \left [ -\Delta_g \scal_g + \scal_g^2-3|\ric_g |_g^2\right ]\ .
\end{equation}
One can impose one of these conditions (constant $A_g$ or constant $Q_g$) in order to fix the infinitely many trace data $T_i$ (except, it turns out, $T_4$) in Theorem \ref{theorem1.1}, leaving finitely many data to be chosen by imposing conditions at infinity.

To see that the condition $A=0$ fixes a unique representative metric $g$ within its conformal class of Bach-flat metrics, consider that if ${\tilde g}:=u^2g$ and $g$ both have scalar curvature $-n(n-1)$, then $u$ must be a positive solution of the Yamabe equation $-\frac{4}{n(n-2)}\Delta_g u + \left ( u^{\frac{4}{(n-2)}}-1\right )u=0$ and $u\to 1$ at conformal infinity. But then $u\equiv 1$ by the maximum principle. (We assume here completeness with no ``inner'' boundary---if one is present, there may sometimes be other solutions for $u$.) Since the condition $Q=6$ fixes the same free data, it also fixes a unique representative metric $g$ within its conformal class of Bach-flat metrics.

It turns out that neither fixing $A_g$ (and thus $\scal_g$) nor fixing $Q_g$ will determine $T_4$. Consider the quantity \cite{Wang, CH, Woolgar}
\begin{equation}
\label{eq1.7}
\mu:=\frac{1}{3!}\trace_{h_0} h^{(4)}(0)-\left \vert \frac{1}{2!}h''(0)\right \vert_{h_0}^2=\frac{1}{3!}T_4-\left \vert \frac{1}{2!}h''(0)\right \vert_{h_0}^2\ .
\end{equation}
When conformal infinity carries a round sphere metric, this quantity is called the \emph{mass aspect function}. In that case, if $g$ is Poincar\'e-Einstein the mass (the integral of $\mu$ over conformal infinity) must vanish \cite{AD}, and indeed so must the mass aspect (e.g., \cite[see the proof of Conjecture 2.7]{Woolgar}). More generally, to select Poincar\'e-Einstein metrics, we must choose the correct conformal class, and this was not completely achieved by choosing data as in Theorem \ref{theorem1.1}. We must in addition impose the condition $T_4=\frac32 |h''(0)|_{h_0}^2$ so that $\mu=0$.\footnote
{There is debate over whether complete metrics can have vanishing mass but nontrivial mass aspect when $n=4$ and $A\ge 0$ (see \cite{CGNP} for further details).}
The following results give two methods for fixing the conformal class.

\begin{corollary}\label{corollary1.2}
Let $\Psi$ be a symmetric $(0,2)$-tensor on conformal infinity with $\trace_{h_0}\Psi=0$, $\divergence_{h_0}\Psi=0$. A formal power series for an asymptotically hyperbolic $4$-metric in normal form with $h(0)=h_0$, $h'(0)=0$, $h''(0)=-2P_{h_0}$, $h'''(0)=\Psi$, and $\frac{1}{3!}\trace_{h_0} \left ( h^{(4)}(0)\right ) =\left \vert P_{h_0}\right \vert_{h_0}^2$ is a formal solution of the system $B_g=0$, $A_g=0$ if and only if it is a formal solution of the Einstein equations.
\end{corollary}

\begin{corollary}\label{corollary1.3}
Let $\Psi$ be a symmetric $(0,2)$-tensor on conformal infinity with $\trace_{h_0}\Psi=0$, $\divergence_{h_0}\Psi=0$. A formal power series for an asymptotically hyperbolic $4$-metric in normal form with $h(0)=h_0$, $h'(0)=0$, $h''(0)=-2P_{h_0}$, $h'''(0)=\Psi$, and $\frac{1}{3!}\trace_{h_0} \left ( h^{(4)}(0)\right ) =\left \vert P_{h_0}\right \vert_{h_0}^2$ is a formal solution of the system $B_g=0$, $Q_g=6$ if and only if it is a formal solution of the Einstein equations.
\end{corollary}

If one does not fix $\frac{1}{3!}\trace_{h_0} \left ( h^{(4)}(0)\right ) =\left \vert P_{h_0}\right \vert_{h_0}^2$ (i.e., $\mu=0$) but one does fix all the other data as in Corollary \ref{corollary1.2} or \ref{corollary1.3}, one obtains for each choice of mass aspect function $\mu$ a formal power series for an asymptotically Poincar\'e-Bach metric. Such a series represents a conformally Poincar\'e-Einstein metric of arbitrary mass.

In \cite{Anderson4}, the Bach-flat condition is studied for constant-scalar-curvature asymptotically de Sitter spacetimes, but using harmonic gauge. The series expansion is in terms of a timelike defining function, and focuses on extracting Einstein metrics, so the expansion is assumed to agree with the usual Fefferman-Graham expansion up to the order of the free data (i.e, to $h^{(2)}(0)$ inclusive, for $n=4$). Modulo these differences, the results are commensurate with ours.

\subsection{Higher bulk dimensions} An intriguing property of Poincar\'e-Einstein metrics is that the free data do not all appear at low orders in the Fefferman-Graham expansion. Instead, in addition to the free data at order zero (the boundary conformal metric), the other free data occur at order $n-1$. This is key to the AdS/CFT correspondence. In the Bach case, it turns out that a similar split between low- and high-order data occurs, but it is not manifest when $n=4$ because of the low dimension.

To explore the phenomenon, we consider a generalization of the Bach tensor to higher dimensions. Of course, the Bach tensor is most naturally defined in $4$ dimensions, where it has vanishing divergence and trace, is a local conformal invariant, and obstructs conformally Einstein metrics, while nontrivially generalizing the Einstein condition (i.e., there are Bach-flat metrics that are not Einstein, but Einstein metrics are Bach-flat). There are many inequivalent generalizations of the Bach tensor for $n>4$ \cite{Bergman}, each preserving some desirable properties of the $4$-dimensional Bach tensor but none preserving them all.\footnote
{It is possible to preserve the desirable properties of the Bach tensor in higher (even bulk) dimensions, at the expense of working with a tensor of higher differential order, specifically the \emph{ambient obstruction tensor}. Helliwell \cite{Helliwell} has studied asymptotically hyperbolic metrics with vanishing obstruction tensor as a generalization of Anderson's boundary regularity studies \cite{Anderson2, Anderson3} to higher (even) dimensions.}

We will present a divergence-free generalization of the Bach tensor to higher dimensions, where we observe the ``splitting'' of the free data into low-order (``Dirichlet'') and high-order (``Neumann') pairs. The fact that there are more high-order data than in the Poincar\'e-Einstein case is curious, and may illustrate a general feature of geometric equations in asymptotically hyperbolic manifolds.

\subsection{Some open questions}
Questions arise from our work, some of which we have considered but are unable to answer. We pose two of the more important ones here, in the hope that others will be able to take them up.

One important question concerns the Lorentzian formulation. What are the free data if the metric is Lorentzian, and how do these data relate to the particle content? Similar questions were studied in physics several years ago. For example, Stelle \cite{Stelle} studied an action for Lorentzian signature metrics in $4$ spacetime dimensions. His action differed from ours in several ways. He had an Einstein-Hilbert term (i.e., a term linear in scalar curvature) as well as terms quadratic in curvatures. He did not include a $|W|^2$ term, citing the Gauss-Bonnet theorem in asymptotically flat spacetime, and he linearized about Minkowski spacetime (thus foregoing a renormalized volume term). He found the following particle content: a graviton, a massive scalar field, and a massive spin-$2$ particle with negative linearized energy.

Another interesting question was raised by Anderson, who asked for a characterization of data at infinity for conformally Poincar\'e-Einstein metrics in arbitrary coordinate gauges \cite[p 463]{Anderson3}. Since we work in a fixed gauge, and study formal expansions only, our work seems not to shed light on this issue.

\subsection{Preview} This paper is organized as follows. In Section 2 we state our conventions and briefly recall the basic theory of asymptotically hyperbolic metrics and Poincar\'e-Einstein metrics. In Section 3 we expand $B_g$ in terms of the tensor $E_g:=\ric_g+(n-1)g$. Section 4 is dedicated to the case of $n=4$. In Section 4.1 we discuss the equation $B^{\perp}=0$ in $n=4$ dimensions (here $\perp$ indicates projection orthogonal to $\partial_x$), while in Section 4.2 we apply the Bianchi identity and obtain a condition on the divergence of the free data $h^{(3)}(0)$. The equation $A=0$ is discussed in Section 4.3. An alternative to fixing the conformal gauge by setting $A=0$ is instead to fix the $Q$-curvature. This is discussed in Section 4.4. The proofs of Theorem \ref{theorem1.1} and Corollary \ref{corollary1.2} then follow quickly from the earlier subsections and are given in Section 4.5. We discuss the $n\ge 5$ case in Section 5.

\subsection{Acknowledgements} AA was supported by a post-doctoral fellowship from the Natural Sciences and Engineering Research Council (NSERC) and an AMS-Simons Travel Grant. The work of EW was supported by an NSERC Discovery Grant RGPIN 203614. Both authors are grateful to the Fields Institute for Research in Mathematical Sciences, where much of this work was carried out, and to the organizers of its 2017 Thematic Programme on Geometric Analysis for a stimulating environment. We are also grateful to the Banff International Research Station for hosting us at its workshop 18W5108 and to C Robin Graham for discussions and helpful comments on an earlier draft.

\section{Preliminaries}\label{section2}
\setcounter{equation}{0}

\subsection{Notation and conventions} As already stated, we use $n=\dim M$ to be the dimension of the bulk manifold $(M,g)$.

We define the rough (or connection) Laplacian to be the trace of the Hessian, i.e., $\Delta_g:=\trace_g \hess=g^{ab}\nabla_a\nabla_b$ for a given Levi-Civita connection $\nabla_g$.

In index notation, we have
\begin{equation}
\label{eq2.1}
\begin{split}
R^a{}_{bcd}{}=&\, \partial_{c}\Gamma_{bd}^a-\partial_{d}\Gamma_{bc}^a + \Gamma_{ce}^a\Gamma_{bd}^e -\Gamma_{de}^a\Gamma_{cd}^e\ , \\
W_{abcd} =&\, R_{abcd} - \frac{1}{(n-2)}\left( g_{ac}R_{bd} -g_{ad}R_{bc} -g_{bc}R_{ad} +g_{bd}R_{ac} \right ) \\
&\, +\frac{1}{(n-1)(n-1)}\left ( g_{ac}g_{bd}-g_{ad}g_{bc}\right )\\
=&\, R_{abcd} - \frac{1}{(n-2)}\left( g_{ac}E_{bd} -g_{ad}E_{bc} -g_{bc}E_{ad} +g_{bd}E_{ac} \right ) \\
&\, +\frac{A}{(n-1)(n-2)}\left ( g_{ac}g_{bd}-g_{ad}g_{bc}\right ) +\frac{n}{(n-2)}\left (  g_{ac}g_{bd}-g_{ad}g_{bc}\right )\ ,
\end{split}
\end{equation}
where $R_{abcd}:=g_{ae}R^e{}_{bcd}$ and we define
\begin{equation}
\label{eq2.2}
\begin{split}
E_{ab}:=&\, R_{ab}+(n-1) g_{ab}\ ,\\
A:=&\, g^{ab}E_{ab}=R+n(n-1)\ .
\end{split}
\end{equation}
We note that $E$ is not the tracefree Einstein tensor (except of course when $A=0$). We also define the Schouten tensor
\begin{equation}
\label{eq2.3}
P_{ab}:= \frac{1}{(n-2)}\left ( R_{ab}-\frac{1}{2(n-1)}Rg_{ab}\right ) \ ,
\end{equation}
and the tracefree Einstein tensor
\begin{equation}
\label{eq2.4}
Z_{ab}:=R_{ab}-\frac{1}{n}R g_{ab}\ .
\end{equation}

Finally, in keeping with standard usage, for a function $f$ depending on a defining function $x$ for conformal infinity, we write $f\in {\mathcal O}(x^p)$ if there are constants $C>0$ and $\epsilon>0$ such that $|f(x)| < Cx^p$ for all $x<\epsilon$. Clearly, if $f\in {\mathcal O}(x^p)$ for some $p>q$, then $f\in {\mathcal O}(x^q)$ as well.

\subsection{Asymptotically hyperbolic metrics} Let $\bar M$ be a compact manifold-with-boundary with interior $M$. A metric $g$ on $M$ is called \emph{conformally compactifiable} if there is a $C^\infty$ metric $\bar{g}$ on $\bar M$ and a positive function $\rho:M\to (0,\infty)$, such that
\begin{equation}
\label{eq2.5}
g = \rho^{-2} \bar{g}
\end{equation}
on $M$, and such that $\rho$ extends smoothly to $\bar M$ with $\rho=0$ and $d\rho\neq 0$ pointwise on $\partial M$.

We refer to $\partial \bar{M}$ as the boundary-at-infinity of $M$. It is sometimes denoted by $\partial_{\infty}{M}$. The conformal equivalence class $[h]$ of $h:=\bar{g}|_{\partial \bar{M}}$ is called the \emph{conformal boundary} of $(M, g)$. We call $\rho$ a defining function for the conformal boundary. We can always arrange that $|d\rho|^2_{\bar g}(\partial \bar{M}) = 1$. If $\bar{g}$ is $C^1$, we can solve the eikonal differential equation $|dx|^2_{\bar{g}} = 1$ in a collar neighbourhood of $\partial \bar{M}$, subject to the boundary condition $x=0$ on $\partial M$. Then $x$ is called a \emph{special defining function} and $(M, g)$ is called \emph{conformally compactifiable and asymptotically hyperbolic}, or simply \emph{asymptotically hyperbolic}. On a neighbourhood of conformal infinity, the metric can then be written in the form of equation \eqref{eq1.3}; equivalently, $dx^2 + h_x$ is a metric in Gaussian normal coordinate form, and $g$ is then said to be in Graham-Lee normal form. By analyzing the formula for the conformal transformation of the curvature, one then sees that the sectional curvatures of an asymptotically hyperbolic metric approach $-1$ as $x\to 0$.

There is some freedom to choose $x$, corresponding to the freedom to choose a conformal representative $h_0$ in $[h]$. We will choose a representative $h_0$ below, so that $x$ will be determined, but the freedom to vary these choices remains. For greater detail, please see \cite{FG2, DGH}.

\subsection{Poincar\'e-Einstein metrics} These are asymptotically hyperbolic Einstein metrics. They obey the negative Einstein equation
\begin{equation}
\label{eq2.6}
E_g:=\ric_g +(n-1)g=0
\end{equation}
on the bulk $n$-dimensional manifold $(M,g)$.

We briefly review the Fefferman-Graham expansion for these metrics. If we insert \eqref{eq1.3} into \eqref{eq2.6}, we obtain
\begin{equation}
\label{eq2.7}
\begin{split}
E_{00}=&\, -\frac12 \tr h_x''+\frac{1}{2x}\tr h_x'+\frac14 \left \vert h_x'\right \vert_{h_x}^2\ , \\
E^{\diamond}=&\, \frac12\left [ \divergence_{h_x} h_x' -d\tr h_x' \right ]\ , \\
E^{\perp}_{h_x}=&\, -\frac12 h_x''+\frac{(n-2)}{2x}h'_x +\frac{1}{2x} h_x \left ( \tr h_x' \right ) +\frac12 h_x'\circ h_x' -\frac14 h_x'\tr h_x'+\ric_{h_x}\ ,
\end{split}
\end{equation}
where $E^{\perp}$ is the tensor on the level sets $x=const$ obtained by orthogonal projection of $E$ onto the tangent spaces of these sets, $E_{00}=E(\partial_x,\partial_x)$, and $E^{\diamond}$ is the covector field on the levels sets of $x$ defined by $E^{\diamond}(\partial_{y^i})=E(\partial_x,\partial_{y^i})$. We denote by $A\circ B$ the contraction whose component form is $(A\circ B)_{ij}:=A_{ik}h^{kl}B_{lk}$.

If one computes the order-$l$ derivative of the above expression for $E^{\perp}$ with respect to $x$, the result is
\begin{equation}
\label{eq2.8}
xh_x^{(l+2)} +(l-n+2)h_x^{(l+1)}-h_x \tr h_x^{(l+1)} = F(h_x,\dots,h_x^{(l)})\ ,\ l=0,1,2,\dots
\end{equation}
where here and in subsequent sections $F$ represents an unspecified function depending only on the listed arguments (and which may change in each subsequent occurrence).
Setting $x=0$ in this equation allows one to compute by iteration the $x$-derivatives of order $1,\dots,n-2$ of $h_x$ at $x=0$ in terms of $h_{(0)}$. When $l=n-2$ the coefficient of $\tf h_x^{(n-1)}$ will vanish. If the tracefree part of $F$ does not vanish under these circumstances, then there is an obstruction to the existence of the Mclaurin expansion of $h_x$ about $x=0$. The nonvanishing terms define the \emph{ambient obstruction tensor} which is of much interest in conformal geometry. The obstruction is avoided by adding logarithmic terms so that we no longer have a Mclaurin expansion for $h_x$, but instead have a polyhomogeneous expansion. In any case, $\tf h_x^{(n-1)}$ is free data and can be freely chosen. Once it has been chosen, the iteration can be restarted and continued to all orders, either as a Mclaurin expansion or a polyhomogeneous expansion, as appropriate. The coefficients of the higher order terms in the expansion will in general depend on both $h_{(0)}$ and $\tf h^{(n-1)}(0)$, but are otherwise completely determined. Two important results easily derived from this iteration procedure are that (i) all the odd derivatives $h_x^{(2l+1)}$ vanish at $x=0$ for $2l+1<n-1$, and (ii) when $n$ is even, $\tr h_x^{(n-1)}$ vanishes at $x=0$.

Because of the second Bianchi identity, one usually focuses attention on $E^{\perp}_{h_x}$ alone, but the vanishing of $E^{\diamond}_{h_x}$ imposes conditions on the divergence of $h_x$ which govern the divergence of certain data. Let $n$ be even. Differentiating $E^{\diamond}_{h_x}$ with respect to $x$ $(n-2)$-times using \eqref{eq2.7}, we obtain that $\divergence_{h_x}h_x^{(n-1)}-d\tr h_x^{(n-1)}$, evaluated at $x=0$, is given by a sum of terms each of which has a factor of the form $h_x^{(2l+1)}\big\vert_{x=0}$, for some $l$ such that $2l+1<n-1$. But in the last paragraph we noted that each odd derivative must vanish. Then $\divergence_{h_x}h_x^{(n-1)}-d\tr h_x^{(n-1)}$ vanishes at $x=0$, and since $\tr h_x^{(n-1)}$ itself vanishes at $x=0$, we conclude that for even $n$ then $\divergence_{h_x}\tf h_x^{(n-1)}\big\vert_{x=0}=0$. In the AdS/CFT correspondence, this allows for the interpretation of $\tf h_x^{(n-1)}\big\vert_{x=0}$ as the vacuum expectation value of the CFT stress-energy tensor \cite{HS, dHSS}. The vanishing of $\tr h_x^{(n-1)}$ means that there is no conformal anomaly (which would break the conformal invariance of the CFT), while the vanishing of $\divergence_{h_x}h_x^{(n-1)}\big\vert_{x=0}$ implies that the appropriate Ward identity is also anomaly-free.

For odd $n$, this analysis determines $\divergence_{h_x}\tf h_x^{(n-1)}\big\vert_{x=0}$ in terms of lower derivatives of $h_x$ at $x=0$, but it need not vanish. Again, for greater detail, please see \cite{FG2, DGH}.

\section{The Bach tensor}
\setcounter{equation}{0}

\subsection{Bach tensor in terms of $E$ and $W$.}
In this section, we record the main formulas used to expand the Bach tensor in a series. We begin by writing
\begin{equation}
\label{eq3.2}
g=\frac{1}{x^2}{\tilde g}\ , \ {\tilde g}=dx^2 \oplus h_x\ .
\end{equation}
We will use ${\tilde \nabla}$ to denote the Levi-Civita connection compatible with ${\tilde g}$. By ${\tilde \nabla}_a E_{bc}$, we mean $\left ( \nabla_{\partial_a} E\right ) (\partial_b, \partial_c)$, and ${\tilde \nabla}^b E_{ab}:={\tilde g}^{bc}{\tilde \nabla}_b E_{ac}$, and of course our coordinates are $x^a\in \{ x^0,y^i\}$, $i\in \{ 1,\dots,n-1\}$; in particular, $x^0\equiv x$.

The formulas are straightforward to derive, but the derivations are often tedious and lengthy calculations, so we include only the main intermediate steps in the derivation. To begin, the Bach tensor can be expanded in terms of $W$, $E$, and $A$.

\begin{lemma}\label{lemma3.1}
\begin{equation}
\label{eq3.1}
\begin{split}
B_{ac}=&\, \frac{1}{(n-2)}\left \{ \Delta E_{ac} - \frac{(n-2)}{2(n-1)}\nabla_a\nabla_c A -\frac{1}{2(n-1)}g_{ac} \Delta A +2W_{dabc}E^{bd}\right .\\
&\, \left . \quad -\frac{n}{(n-2)}\left [ E_a{}^bE_{bc}-\frac{A}{(n-1)}E_{ac}\right ] +\frac{1}{(n-2)} \left [ |E|^2 -\frac{A^2}{(n-1)}\right ] g_{ac}\right . \\
&\, \left . \quad + nE_{ac}-g_{ac}A\right \}\ .
\end{split}
\end{equation}
Then
\begin{equation}
\label{eq3.3}
\begin{split}
\Delta E_{ac} \equiv &\, \Delta_g E_{ac}\\
= &\, x^2 \Delta_{\tilde g} E_{ac} +x\left [ 6 {\tilde \nabla}_0 E_{ac} +2\delta^0_a {\tilde \nabla}^b E_{bc}+2\delta_c^0 {\tilde \nabla}^bE_{ab}-2{\tilde \nabla}_a E_{0c} -2{\tilde \nabla}_c E_{0a} -n{\tilde \nabla}_0 E_{ac}\right ] \\
&\, -2(n-2) E_{ac} +2{\tilde g}_{ac} E_{00} -n\delta_a^0 E_{0c} -n\delta_c^0 E_{0a} +2\delta_a^0\delta_b^0 \left ( E_{00}+h^{ij}E_{ij}\right )\ .
\end{split}
\end{equation}
\end{lemma}

\begin{proof} These results are by direct and simple, if tedious, computation. To obtain \eqref{eq3.1}, simply plug \eqref{eq2.6} into \eqref{eq1.1} and compute using the second Bianchi identity. To obtain \eqref{eq3.3}, note that the connection coefficients ${\tilde \Gamma}^a_{bc}$ of ${\tilde g}_{ab}$ are related to those of $g_{ab}$ (denoted $\Gamma^a_{bc}$) by
\begin{equation}
\label{eq3.4}
\Gamma^a_{bc}={\tilde \Gamma}^a_{bc}-\frac{1}{x}\left ( \delta^0_b\delta^a_c +\delta^0_c\delta^a_b -\delta^a_0{\tilde g}_{bc} \right )\ .
\end{equation}
The usual expansion for a connection in terms of its coefficients yields
\begin{equation}
\label{eq3.5}
\nabla_a E_{bc} ={\tilde \nabla}_a E_{bc} +\frac{1}{x} \left ( 2\delta^0_a E_{bc} +\delta^0_b E_{ac} +\delta^0_c E_{ab} -{\tilde g}_{ab} E_{0c} -{\tilde g}_{ac}E_{0b} \right )\ .
\end{equation}
Now differentiate once more by applying $\nabla_a$ to \eqref{eq3.5} and use that $\Delta_g E_{ac}= g^{bd}\left ( \nabla_b\nabla_d E_{ac}\right )$. This is lengthy but simple and we omit the details.
\end{proof}

It will be useful to expand equation \eqref{eq3.1} componentwise. The non-vanishing Christoffel symbols of the Levi-Civita connection of ${\tilde g}_{ij}$ in the coordinates $x^a\in \{x^0=x,y^i\}$ are
\begin{equation}
\label{eq3.6}
\begin{split}
{\tilde \Gamma}^0_{ij} = &\, -\frac12 h_{ij}' =: K_{ij}\ ,\\
{\tilde \Gamma}^i_{0j} = &\, {\tilde \Gamma}^i_{j0} = \frac12 h^{ik}h_{jk}'=-h^{ik}K_{jk}=:K^i{}_j\ ,\\
{\tilde \Gamma}^i_{jk} = &\, \Xi^i_{jk}\ ,
\end{split}
\end{equation}
where the $\Xi^i_{jk}$ are the Christoffel symbols of the Levi-Civita connection $D=D_x$ compatible with $h_x$ on each constant-$x$ slice. Then we easily compute that
\begin{equation}
\label{eq3.7}
\begin{split}
{\tilde \nabla}_0 E_{00} =&\, E_{00}'\ , \\
{\tilde \nabla}_i E_{00} =&\, D_i E_{00} +2K_i{}^kE_{0k}\ ,\\
{\tilde \nabla}_0 E_{0i} =&\, E_{0i}' +K_i{}^kE_{0k}\ ,\\
{\tilde \nabla}_i E_{0j} =&\, D_i E_{0j} +K_i{}^k E_{jk}-K_{ij}E_{00}\ ,\\
{\tilde \nabla}_0 E_{ij} =&\, E_{ij}' + K_i{}^kE_{jk} +K_j{}^kE_{ik}\ ,\\
{\tilde \nabla}_i E_{jk} =&\, D_i E_{jk} -K_{ij}E_{0k}-K_{ik}E_{0j}\ .
\end{split}
\end{equation}
Differentiating these expressions once more and summing, one obtains the expressions
\begin{equation}
\label{eq3.8}
\begin{split}
\left ( \Delta_g E\right )_{00} = &\, x^2 \left \{ E_{00}'' +\Delta_h E_{00} -HE_{00}' -2|K|_h^2 E_{00} \right . \\
&\,\qquad \left . +2D^i\left ( K_i{}^j E_{0j}\right ) +2K^{ij}D_i E_{0j} +2K^i{}_j K^{jk}E_{ik} \right \} \\
&\, -x \left \{ (n-6)E_{00}' -4D^iE_{0i} -4K^{ij}E_{ij} +4HE_{00}\right \} -4(n-2) E_{00} +2h^{ij}E_{ij}\ ,\\
\left ( \Delta_g E\right )_{0i} = &\, x^2 \left \{ E_{0i}'' +\Delta_hE_{0i} -HE_{0i}' +2K_i{}^j E_{0j}' + (K_i{}^j)'E_{0j} -2K_i{}^jK_j{}^kE_{0k} -|K|_h^2 E_{0i} \right .\\
&\, \qquad \left . -HK_i{}^jE_{0j} +2K^{jk}D_j E_{ik} +(D_j K^{jk})E_{ik} -2K_i{}^jD_j E_{00}-(D^j E_{ij}) E_{00} \right \} \\
&\, -x \left \{ (n-6) E_{0i}' +nK_i{}^j E_{0j} -2D^j E_{ij} +2HE_{0i} +2D_i E_{00} \right \} -(3n-4) E_{0i}\ , \\
\left ( \Delta_g E\right )_{ij} = &\, x^2 \left \{ E_{ij}'' +\Delta_h E_{ij} -HE_{ab}' +2K_i{}^kE_{jk}' +2K_j{}^k E_{ik}' -HK_i{}^kE_{jk} -HK_j{}^k E_{ik} \right .\\
&\, \left . \qquad +(K_i{}^k)'E_{jk} +(K_j{}^k)'E_{ik} +2K_i{}^kK_j{}^lE_{kl} -2K_{ik}D^k E_{0j} -2K_{jk}D^k E_{0i} \right . \\
&\, \qquad \left . -(D^kK_{ik})E_{0j}-(D^kK_{jk})E_{0i}+2K_i{}^kK_{jk} E_{00} \right \} \\
&\, -x \left \{ (n-6) E_{ij}' +(n-8) \left ( K_i{}^kE_{jk} +K_j{}^kE_{ik}\right ) -2 \left ( D_iE_{0j}+D_jE_{0i}\right ) -4K_{ij} E_{00} \right \} \\
&\, -2(n-2) E_{ij} +2h_{ij} E_{00}\ .
\end{split}
\end{equation}
In the above, indices are raised with $h^{-1}$, denoted as usual by $h^{ij}$, and $H:=\trace_h K =h^{ij}K_{ij}$ denotes the mean curvature of level sets of $x$. We need also that
\begin{equation}
\label{eq3.9}
\begin{split}
(\hess A)_{00}=&\, A''+\frac{1}{x}A' \ , \\
(\hess A)_{0i}=&\, D_i (A') + K_i{}^k D_k A +\frac{1}{x}D_i A \ , \\
(\hess A)_{ij}=&\, D_iD_j A -K_{ij}A' -\frac{1}{x}h_{ij} A' \ ,\\
\implies \Delta_g A =&\, x^2 \left [ A'' -HA' + \Delta_h A \right ] -(n-2)x A'\ ,
\end{split}
\end{equation}
where \eqref{eq2.7} yields
\begin{equation}
\label{eq3.10}
\begin{split}
A=&\, -x^2\tr h_x'' +(n-1)x \tr h_x' +\frac34 x^2 \left \vert h_x'\right \vert^2_{h_x} -\frac{x^2}{4} \left ( \tr h_x'\right )^2 +x^2 \scal_{h_x}\\
=&\, x^2 \left ( 2H' -|K|_{h_x}^2-H^2 +\scal_{h_x} \right ) -2(n-1)xH \ .
\end{split}
\end{equation}

Putting this all together, we have
\begin{equation}
\begin{split}
\label{eq3.11}
B_{00} = &\, \frac{x^2}{(n-2)} \left \{ \frac12 {E_{00}}'' -\frac12 \tr ({E^{\perp}}'') -\frac{(2n-3)}{2(n-1)}HE_{00}' +\frac{1}{2(n-1)}H\tr ( {E^{\perp}}' )\right . \\
&\, \left . -2K^{ij}E_{ij}'-(K^{ij})'E_{ij} +\Delta_h E_{00}-\frac{1}{2(n-1)}\Delta_h {\tilde A} +2D_i (K^{ij} E_{0j}) +2K^{ij}D_iE_{0j}\right . \\
&\, \left . +2K^{ij}K_i{}^kE_{jk} +\frac{1}{(n-1)}HK^{ij}E_{ij} -2|K|^2_h E_{00} -E_{00}^2 -|E_{0i}|_h^2 \right . \\
&\, \left . +\frac{1}{(n-1)}E_{00}\tr E^{\perp}
+ \frac{1}{(n-2)} \left [ | E^{\perp} |_h^2 -\frac{1}{(n-1)} \left ( \tr E^{\perp} \right )^2 \right ] \right \}\\
&\, +\frac{x}{(n-2)} \left \{ -(n-4) E_{00}' -2\tr ( {E^{\perp}}') +4D^jE_{0j} -\frac{(6n-7)}{(n-1)}HE_{00}\right . \\
&\, \left . +\frac{1}{(n-1)} H \tr E^{\perp} \right \} \\
&\, -3E_{00} +\frac{2x^4}{(n-2)} W_{0i0j}h^{ik}h^{jl}E_{kl}\ ,
\end{split}
\end{equation}
where we write ${\tilde A}:=\trace_{\tilde g} E \in{\mathcal O}(1/x)$. Continuing, we have
\begin{equation}
\label{eq3.12}
\begin{split}
B_{0i}= &\, \frac{x^2}{(n-2)} \left \{ E_{0i}'' +2K_i{}^jE_{0j}'-HE_{0i}' -\frac{(n-2)}{2(n-1)}D_i \left [ E_{00}' +\tr ({E^{\perp}}')\right ] +\Delta_h E_{0i} \right .\\
&\, \left . +D^k\left ( K_k{}^j E_{ij}-K_{ik}E_{00} \right ) +K^{jk}D_j E_{ik} -K_i{}^j D_j E_{00} -\frac{(n-2)}{2(n-1)} \left [ 2D_i (K^{jk}E_{jk}) \right . \right .\\
&\, \left . \left . +K_i{}^j D_j E_{00} + K_i{}^j D_j (\tr E^{\perp} ) \right ] +K_{ij}' h^{jk}E_{0k} -HK_i{}^jE_{0j} -|K|_h^2 E_{0i}\right . \\
&\, \left . -\frac{n}{(n-2)} \left [ E_{00}E_{0i} +E_{0j}h^{jk}E_{ki} \right ] +\frac{n}{(n-1)(n-2)} \left ( E_{00}+\tr E^{\perp} \right ) E_{0i}\right \}\\
&\, +\frac{x}{(n-2)}\left \{ -(n-6)E_{0i}' -2D_i E_{00} +2D^jE_{ij} -\frac{3(n-2)}{2(n-1)} D_i \left ( E_{00} +\tr E^{\perp} \right )\right . \\
&\, \left . -nK_i{}^jE_{0j} -4HE_{0i} \right \}\\
&\, -2E_{0i} +\frac{2x^4}{(n-2)}\left ( W_{0jik}h^{jl}h^{kp}E_{lp}-W_{0j0i}h^{jk}E_{0k}\right ) \ ,
\end{split}
\end{equation}
and finally
\begin{equation}
\label{eq3.13}
\begin{split}
B_{ij}= &\, \frac{x^2}{(n-2)} \left \{ E_{ij}'' -\frac{1}{2(n-1)} \left [ E_{00}'' +\tr ({E^{\perp}}'')\right ]h_{ij}-HE_{ij}' +2K_i{}^kE_{jk}' +2K_j{}^k E_{ik}' \right .\\
&\, \left . +\frac{1}{2(n-1)} \left [ (n-2) K_{ij} +Hh_{ij} \right ]E_{00}' +\frac{1}{2(n-1)} \left [ (n-2)K_{ij}+Hh_{ij} \right ] \tr ({E^{\perp}}')\right . \\
&\, \left . -\frac{2}{(n-1)}K^{kl}E_{kl}'h_{ij} -\frac{n}{(n-2)} \left [ E_{ik}h^{kl}E_{jl} -\frac{1}{n} \left ( E_{00}^2 +| E^{\perp}|_h^2 \right ) h_{ij} \right ] \right . \\
&\, \left . +\frac{n}{(n-1)(n-2)} \left [ E_{00} +\tr E^{\perp} \right ] \left [ E_{ij} -\frac{1}{n}E_{00} h_{ij} -\frac{1}{n} (\tr E^{\perp} ) h_{ij}\right ] \right \}\\
&\, +\frac{x}{(n-2)} \left \{ -(n-6)E_{ij}' -(n-4)\left ( K_i{}^kE_{jk} +K_j{}^kE_{ik} \right ) +4K_{ij}E_{00} -2HE_{ij} \right .\\
&\, \left . +\frac{(n-4)}{(n-1)} \left [ E_{00}' h_{ij} +\tr ({E^{\perp}}')h_{ij} \right ] +\frac{2(n-4)}{(n-1)}K^{kl}E_{kl} h_{ij}\right .\\
&\, \left . +\frac{1}{(n-1)} \left [ (n-2)K_{ij}+Hh_{ij}\right ] \left [ E_{00}+\tr E^{\perp} \right ] \right \}\\
&\, -\frac{(n-4)}{(n-2)} \left [ E_{ij} -\frac{1}{(n-1)} (\tr E^{\perp})h_{ij}\right ] +\frac{3}{(n-1)}E_{00}h_{ij} \\
&\, +\frac{x^4}{(n-2)}\left [ W_{ikjl}h^{kp}h^{lq}E_{pq}+2W_{ikj0}h^{kl}E_{0l}+W_{i0j0}E_{00} \right ] \ .
\end{split}
\end{equation}

Despite their lengths, the above expressions have a simple structure, owing at least in part to the quasilinearity of the Bach tensor as a function of the metric. For example, the tensor $E$ of an asymptotically hyperbolic metric is ${\mathcal O}(x)$ in all dimensions, while its normal form components can diverge as ${\mathcal O}(1/x)$. For the Bach tensor we can now read off from the above expressions the comparable result.

\begin{lemma}\label{lemma3.2}
The normal-form components of the Bach tensor of an asymptotically hyperbolic $n$-manifold admit a $C^0$ extension to conformal infinity, and vanish at conformal infinity when $n=4$. We have
\begin{equation}
\label{eq3.14}
|B|_g\in \begin{cases} {\mathcal O}(x^2),& n>4 \\ {\mathcal O}(x^3),& n=4 \end{cases}\ .\
\end{equation}
\end{lemma}

\begin{proof}
In deriving equations \eqref{eq3.11}, \eqref{eq3.12}, and \eqref{eq3.13}, we have not expanded the Weyl tensor contribution to the Bach tensor. To establish the lemma, such an expansion is not necessary. To see this, observe first that the conformal properties of the Weyl tensor are such that $W^g_{abcd}=\frac{1}{x^2}W^{\tilde g}_{abcd}$. Now $\left ( W^{\tilde g}\right )^a{}_{bcd}\in {\mathcal O}(1)$, so the components of $W^{\tilde g}$ with respect to a normal-form basis $\{ \partial_0,\partial_i \}$ obey $W^{\tilde g}_{abcd}\in{\mathcal O}(1)$ as well (lowering the index with ${\tilde g}$). Hence $W^g_{abcd}\in {\mathcal O}(x^{-2})$. Further, $h^{ij}\in {\mathcal O}(1)$ and, by equations \eqref{eq2.7}, $E_{ab}\in {\mathcal O}(1/x)$. Thus, any product of the form $W^g*h^{-1}*E^g$ or $W^g*h^{-1}*h^{-1}*E^g$ (with all indices lowered in $W$) has components that are (at worst) ${\mathcal O}(x^{-3})$. But each time such a term appears in equations \eqref{eq3.11}, \eqref{eq3.12}, and \eqref{eq3.13}, it appears with coefficient $x^4$, and hence the components of these terms in the normal-form basis vanish at least as ${\mathcal O}(x)$.

We simply substitute equations \eqref{eq2.7} into equations \eqref{eq3.11}--\eqref{eq3.13}. Together with the fact that the Weyl tensor term $W_{dabc}E^{bd}$ in \eqref{eq3.1}, expressed in a normal form basis, is ${\mathcal O}(x)$, straightforward cancellation of terms now leads directly to
\begin{equation}
\label{eq3.15}
\begin{split}
B_{00}\in &\, {\mathcal O}(x)\ , \\
B_{0i}\in &\, {\mathcal O}(x)\ , \\
B_{ij}=&\, \frac{(n-4)}{(n-2)} \left \{ (n-3) \left [ K_{ij}' -\frac{1}{(n-1)}H'h_{ij}+\frac{2}{(n-1)}|K|_h^2 h_{ij} \right ] \right .\\
&\, \left . \qquad +(n-4) \left [ K_i{}^kK_{jk}-\frac{1}{(n-1)} |K|_h^2 h_{ij} \right ] -Z_{ij}^h +HK_{ij} -\frac{1}{(n-1)}H^2 h_{ij} \right \} \\
&\, +{\mathcal O}(x)\ ,
\end{split}
\end{equation}
where we write $h$ for the boundary metric $h:= h_x\big \vert_{x=0}$ and write $Z^h:=\tf \ric_h$.

Thus when $n=4$, the normal-form components of $B$ are of order $x$. If the components of a $(0,2)$-tensor in the normal-form basis are in ${\mathcal O}(x^p)$, obviously the tensor norm of that tensor is in ${\mathcal O}(x^{p+2})$, so $|B|_g\in {\mathcal O}(x^3)$.
\end{proof}


\subsection{The Bianchi identity}
In the sequel we will have very little need of the expansions of $B^{\diamond}$ (i.e., $B_{0i}$) and $B_{00}$. Instead, we will solve the equation $B^{\perp}=0$ and use the vanishing of the divergence of $B$ to show that the remaining components vanish. As well, $B_{00}$ will vanish simply because the Bach tensor is traceless.

The vanishing of the divergence of the Bach tensor yields
\begin{equation}
\label{eq3.18}
B_{0i}'-\left [ H+\frac{(n-2)}{x}\right ]B_{0i} = -D^jB_{ij}\ .
\end{equation}
Obviously when $B_{ij}=0$ this becomes a homogeneous linear system for $B_{0i}$, admitting the trivial solution.

\begin{proposition}\label{proposition3.3}
Assume that $B_{ij}^{(\alpha)}(0)=0$ for all $0\le\alpha\le\beta$. Then $B_{00}^{(\alpha)}(0)=0$ and $B_{0i}^{(\alpha +1)}(0)=0$ for all $0\le\alpha\le\beta$, where in the case of $\beta\ge n-2$ we must further assume that $B_{0i}^{(n-2)}(0)=0$.
\end{proposition}

\begin{proof}
As mentioned above, if $B_{ij}^{(\alpha)}(0)=0$ for all $0\le\alpha\le\beta$ then $B_{00}^{(\alpha)}(0)=0$  because $B$ is tracefree.

For $B_{0i}$, we have from the proof of Lemma \ref{lemma3.2} that $B_{0i}\in {\mathcal O}(x)$ and $B_{ij}\in {\mathcal O}(1)$. Then expand $B_{0i}=\sum\limits_{\beta=1}^{\infty}b_{i(\beta)}x^{\beta}$, $H= \sum\limits_{\beta=0}^{\infty} h_{(\beta)}x^\beta$, and $D^jB_{ij}=\sum\limits_{\beta=2}^{\infty}c_{i(\beta)}x^{\beta}$. Note that $B_{0i}^{(n-2)}(0)=0$ if and only if $b_{i(n-2)}=0$. Then it is an easy exercise to expand \eqref{eq3.18} and obtain
\begin{equation}
\label{eq3.19}
\begin{split}
&\, \sum\limits_{\beta=1}^{\infty}\left [ \beta -(n-2) \right ] x^{\beta-1} -\sum\limits_{\beta=1}^{\infty}\left ( \sum_{\alpha=1}^{\beta}b_{i(\alpha)}h_{(\beta-\alpha)}\right )x^{\beta} -\sum\limits_{\beta=0}^{\infty}c_{i(\beta)}x^{\beta} =0 \\
\implies &\, \sum\limits_{\beta=0}^{\infty}\left [ \beta -(n-3) \right ] x^{\beta} -\sum\limits_{\beta=1}^{\infty}\left ( \sum_{\alpha=1}^{\beta}b_{i(\alpha)}h_{(\beta-\alpha)}\right )x^{\beta} -\sum\limits_{\beta=0}^{\infty}c_{i(\beta)}x^{\beta} =0\ .
\end{split}
\end{equation}
Equating coefficients of powers of $x$, we have
\begin{equation}
\label{eq3.20}
\begin{split}
b_{i(1)}=&\, -\frac{1}{(n-3)}c_{i(0)}\ , \\
\left [ \beta -(n-3)\right ] b_{i(\beta +1)}=&\, c_{i(\beta)}+\sum\limits_{\alpha}^{\beta} b_{i(\alpha)}h_{(\beta-\alpha)}\ .
\end{split}
\end{equation}
It follows by induction that $b_{i(\beta +1)}=0$ and so $B_{0i}^{(\beta +1)}(0)=0$. One sees from the left-hand side of \eqref{eq3.20} that the induction pauses when $\beta=n-3$, but then the assumption $B_{0i}^{(n-2)}(0)=0$ fulfils the inductive hypothesis and the induction can be restarted and continued arbitrarily.
\end{proof}


\section{Proof of the main theorem}
\setcounter{equation}{0}

\subsection{The equation for $B_{ij}$} We set $n=4$. Then equation \eqref{eq3.13} becomes
\begin{equation}
\label{eq4.1}
\begin{split}
B_{ij}= &\, \frac12 x^2 \left \{ E_{ij}'' -\frac{1}{6} \left [ E_{00}'' +\tr ({E^{\perp}}'')\right ]h_{ij}-HE_{ij}' +2K_i{}^kE_{jk}' +2K_j{}^k E_{ik}' \right .\\
&\, \left . +\frac{1}{3} \left [  K_{ij} +\frac12 Hh_{ij} \right ]E_{00}' +\frac{1}{3} \left [ K_{ij}+ \frac12 Hh_{ij} \right ] \tr ({E^{\perp}}')\right . \\
&\, \left . -\frac{2}{3}K^{kl}E_{kl}'h_{ij} -2 E_{ik}h^{kl}E_{jl} +\frac{1}{2} \left ( E_{00}^2 +| E^{\perp}|_h^2 \right ) h_{ij} \right . \\
&\, \left . +\frac23 \left [ E_{00} +\tr E^{\perp} \right ] \left [ E_{ij} -\frac{1}{4}E_{00} h_{ij} -\frac{1}{4} (\tr E^{\perp} ) h_{ij}\right ] \right \}\\
&\, +\frac12 x \left \{ 2E_{ij}' +4K_{ij}E_{00} -2HE_{ij} +\frac23 \left [ K_{ij}+\frac12 Hh_{ij}\right ] \left [ E_{00}+\tr E^{\perp} \right ] \right \}\\
&\, +E_{00}h_{ij}+\frac12 x^4\left [ W_{ikjl}h^{kp}h^{lq}E_{pq}+2W_{ikj0}h^{kl}E_{0l}+W_{i0j0}E_{00} \right ] \ .
\end{split}
\end{equation}

In view of Lemma \ref{lemma3.2}, the above expression can be expanded as a power series in $x$. If one substitutes \eqref{eq2.7} into \eqref{eq4.1}, one obtains an expression that is perfectly regular at $x=0$---indeed, with vanishing constant term---despite the fact that the expression for $E$ in \eqref{eq2.7} has some divisions by $x$. In particular, let $\LWT$ denote a sum of \emph{lower weight terms}. These are terms that are regular at $x=0$ and have the form of a (possibly) derivative-dependent coefficient $C(h_x,h_x',\dots,h_x^{(p)})$ multiplying a nonnegative power of $x$, say $x^q$. The \emph{weight} is defined to be the order of the highest $x$-derivative of $h_x$ upon which $C$ depends minus the power of $x$ multiplying the term; i.e., the weight is $p-q$. For example, the weight of the term $-\frac14 x^2 \tf h_{ij}^{(4)}(x)$ is $4-2=2$, while a term such as $x^2\left ( \tr h'\right )\tf h_{ij}'$ would have weight $1-2=-1$. Then we have the following.

\begin{lemma}\label{lemma4.1}
For a metric of the form \eqref{eq1.3} with $h_x=\sum_{i=0}^{4}h_{(i)}x^i+{\mathcal O}(x^5)$ then
\begin{equation}
\label{eq4.2}
B_{ij}(x)=-\frac14 x^2 \tf h_{ij}^{(4)}(x)+\LWT\ .
\end{equation}
\end{lemma}

\begin{proof}
Simply plug \eqref{eq2.7} into \eqref{eq4.1}. While the resulting expression is very lengthly, one can eliminate most terms immediately by observing that the highest weight contributions must arise from the linear terms $-\frac12 x^2 E_{ij}''$, $-\frac{1}{6} \left [ E_{00}'' +\tr ({E^{\perp}}'')\right ]h_{ij}$, $xE_{ij}'$, and $E_{00}h_{ij}$. Expanding these terms using \eqref{eq2.7} yields the result.
\end{proof}

\begin{lemma}\label{lemma4.2}
Let $B_{ij}(x)=0$ and $n=4$. Then for some $h_0$-tracefree tensor $F$ on $\partial_{\infty}M$ and any $s\ge 4$ we may write
\begin{equation}
\label{eq4.3}
\tracefree_{h_0}h_{ij}^{(s)}(0)=F_{ij}(h_0,h'(0),\dots,h^{(s-1)}(0))\ .
\end{equation}
\end{lemma}

\begin{proof}
Equation \eqref{eq4.2} with $B_{ij}(x)=0$ implies that
\begin{equation}
\label{eq4.4}
\frac14 x^2 \tf h_{ij}^{(4)}(x)=\LWT\ .
\end{equation}
If one differentiates the left-hand side $r$-times, with $r\ge 2$, and sets $x=0$, one obtains $\frac14r(r-1)\tracefree_{h_0}h^{(r+2)}(0)$ plus terms of lower differential order.

On the right-hand side, consider a term of weight $w:=p-q$ for $p$ and $q$ as described immediately before Lemma \ref{lemma4.1}. If $r<q$, a factor of $x$ remains after differentiation, so the term vanishes when we set $x=0$. Hence take $r\ge q$. Then the term contributes as $\frac{r!}{(r-q)!}\frac{\partial^{r-q}}{\partial x^{r-q}}\big \vert_{x=0} C(h_x,h_x',\dots,h_x^{(p)})$. Thus, the highest derivative that can arise from this term is $h^{(r-q+p)}(0)$. Now $r-q+p=r+w<r+2$ since $w<2$.

Combining both sides, we have that $\tracefree_{h_0}h^{(r+2)}(0)$ equals a sum of terms that depend on no derivative higher than $h^{(r+1)}(0)$. Now set $s=r+2$.
\end{proof}

This lemma does not determine the trace of $h^{(r)}(0)$ for any order $r$. It does, however, show that one can determine all the coefficients in a formal power series solution of $B_{ij}=0$ in the case of an asymptotically hyperbolic $4$-dimensional bulk manifold in terms of given data $h_0\equiv h(0)$, $h'(0)$, $h''(0)$, and $h'''(0)$ at the conformal boundary, if one is also given as data the traces $\trace_{h_0}h^{(r)}(0)$ for all $r$. There are no obstructions, so it is not necessary to augment the power series with logarithmic terms.

\begin{corollary}\label{corollary4.3}
If $B_{ij}(x)=0$ then there is an $h_0$-tracefree tensor $G$ on $\partial_{\infty}M$ such that
\begin{equation}
\label{eq4.5}
\tracefree_{h_0}h_{ij}^{(s)}(0) =G(h_0,h'(0),h''(0),h'''(0),\trace_{h_0}h^{(4)}(0), \dots,\trace_{h_0}h^{(s-1)}(0)) \ ,\ s\ge 4\ .
\end{equation}
\end{corollary}

\begin{proof}
Immediate from Lemma \ref{lemma4.2} by induction on $s$.
\end{proof}

\subsection{The equation for $B_{0i}$}
We know from Proposition \ref{proposition3.3} that, in the presence of the condition $B_{ij}(x)=0$, then the series expansion $B_{0i}(x)$ is determined by the divergence identity except for the coefficient of the order $x^{n-2}$ term. Here we have $n=4$. Then the only additional information to be learned from solving the $B_{0i}(x)=0$ equation directly is the condition(s) under which $B_{0i}''(0)$ will vanish. The next result shows that the condition $B_{0i}''(0)=0$ imposes a condition on the data $h^{(3)}(0)$ which, in the Poincar\'e-Einstein setting, has an important interpretation in the AdS/CFT correspondence. The same result in the Poincar\'e-Einstein case is essential for the AdS/CFT correspondence, because it allows for the interpretation of $h^{(3)}(0)$ (or, for a $2n$-dimensional bulk, $h^{(2n-1)}(0)$) as the vacuum expectation value of the stress-energy for a conformal field theory defined on $\partial_{\infty}M$.

\begin{proposition}\label{proposition4.4} Let $n=4$ and choose $h_{ij}'(0)=0$. Then $B_{0i}''(0)=0 \Leftrightarrow \divergence_{h_0}\tracefree_{h_0}h^{(3)}(0)=0$.
\end{proposition}

\begin{proof}
Explicit computation beginning with \eqref{eq3.12} and using \eqref{eq2.7} yields
\begin{equation}
\label{eq4.6}
B_{0i}=\left [ \frac12 x^2 \divergence_{h_x} \left ( \tf h_x^{(3)} \right ) -\frac12 (n-4) x \divergence_{h_x} \left ( \tf h_x'' \right )+ \LWT\right ]_i\ .
\end{equation}
Now set $n=4$, differentiate twice with respect to $x$, and set $x=0$. Upon taking two $x$-derivatives of \eqref{eq3.12}, one can see by inspection (using as well \eqref{eq2.7}) that each term arising from twice differentiating the terms denoted $\LWT$ either contains a factor of $K_{ij}$ or has a coefficient of $x$ or $x^2$. Hence these terms vanish upon setting $x=0$ and then $K_{ij}:=-\frac12 h_{ij}'(0)=0$. Thus we obtain
\begin{equation}
\label{eq4.7}
0=B_{0i}''(0)=D^k \left ( \tf h_{ik}^{(3)}(0) \right )\ .
\end{equation}
\end{proof}

\subsection{The condition $A=0$}
This condition is imposed in Corollary \ref{corollary1.2}. Its role is to choose a unique conformal representative within a conformal class $[g]$ of solutions of $B=0$.

We may compute from \eqref{eq2.7} that
\begin{equation}
\label{eq4.8}
A:= \trace_g E = -x^2 \tr h_x''+(n-1)x\tr h_x' +\frac34 x^2 |h_x'|_{h_x}^2-\frac{x^2}{4}\left ( \tr h_x'\right )^2 +x^2\scal_{h_x}\ .
\end{equation}
It is convenient to write the condition $A(x)=0$ as
\begin{equation}
\label{eq4.9}
0 = -x \tr h_x''+(n-1)\tr h_x' +\frac34 x |h_x'|_{h_x}^2-\frac{x}{4}\left ( \tr h_x'\right )^2 +x\scal_{h_x} \ ,
\end{equation}
from which it follows immediately that
\begin{equation}
\label{eq4.10}
\trace_{h_0} h'(0)=0\ .
\end{equation}
If we further assume that $h'(0)=0$, then we can differentiate \eqref{eq4.9} once and set $x=0$ to obtain
\begin{equation}
\label{eq4.11}
\trace_{h_0} h''(0)=-\frac{1}{(n-2)}\scal_{h_0}\equiv -2\trace_{h_0} P_{h_0}\ .
\end{equation}
If one differentiates \eqref{eq4.9} twice, sets $x=0$, and uses $h'(0)=0$, then one obtains
\begin{equation}
\label{eq4.12}
\trace_{h_0} h'''(0)=0\ .
\end{equation}
In general, if one differentiates \eqref{eq4.9} $r\ge 1$ times with respect to $x$ and sets $x=0$, one obtains
\begin{equation}
\label{eq4.13}
0=(n-1-r)\trace_{h_0}h^{(r+1)}(0)+F_n(h_0,h'(0),\dots,h^{(r)}(0))
\end{equation}
for some function $F_n$ that depends on the dimension $n$. When $r=n-1$, one see from this that $\trace_{h_0}h^{(n)}(0)$ is undetermined, and that there are no solutions unless $F_n(h_0,h'(0),\dots,h^{(n-1)}(0))=0$ as well.

\begin{proposition}\label{proposition4.5}
For $n=4$, $\Psi$ a smooth symmetric $h_0$-tracefree $(0,2)$-tensor such that $\divergence_{h_0}\Psi=0$, and $T_4$ some arbitrary function, choose $\tracefree_{h_0} h'(0)=0$, $\tracefree_{h_0} h''(0)=-2Z_{h_0}$, $\tracefree_{h_0} h'''(0)=\Psi$, and $\trace_{h_0}h^{(4)}(0)=T_4$. If $B_g=0$ and $A_g=0$, then $\trace_{h_0}h^{(k)}(0)$ is uniquely determined for all $k$.
\end{proposition}

\begin{proof} This is obvious from the above expressions \eqref{eq4.10}--\eqref{eq4.13} and Theorem \ref{theorem1.1}, provided equation \eqref{eq4.13} has a solution; i.e., provided $F_4=0$. With the chosen data, we have from \eqref{eq4.10}--\eqref{eq4.12} that
\begin{equation}
\label{eq4.14}
h'(0)=0\ ,\ h''(0)=-2P_{h_0}\ ,\ \tracefree_{h_0}h'''(0)=\Psi\ .
\end{equation}
Then we obtain
\begin{equation}
\label{eq4.15}
F_4=-6 \left \vert P_{h_0}\right \vert_{h_0}^2 -6\left ( \trace_{h_0} P_{h_0})\right )^2 + 6\left [ R_{h_0}^{ij}\left ( P_{h_0}\right )_{ij}-D^iD^j \left ( P_{h_0}\right )_{ij}+\Delta_{h_0} \left ( \trace_{h_0}P_{h_0})\right) \right ] \ .
\end{equation}
The Bianchi identity ensures that $-D^iD^j \left ( P_{h_0}\right )_{ij}+\Delta_{h_0} \left ( \trace_{h_0}P_{h_0})\right) =0$, and it is a simple matter to check that the first three terms on the right of \eqref{eq4.15} sum to zero as well, so $F_4$ vanishes as claimed.
\end{proof}

This is, of course, not an accident. The conditions $h'(0)=0$, $h''(0)=-2P_{h_0}$, $\trace_{h_0}h'''(0)=0$, imply that $g$ is a $4$-dimensional APE (asymptotically Poincar\'e-Einstein) metric. There is no obstruction to power series in $x$ for such metrics when the bulk dimension $n$ is even, meaning that for this data the Einstein equations can be solved to order $n$ inclusive (and indeed to any order in $x$). Therefore, we can always solve the equation $A=0$ to order $n$ inclusive (for $n$ even), given data for an APE metric. Beyond order $n$ the coefficient on the left-hand side of equation \eqref{eq4.13} never vanishes, so no obstruction to a recursive solution arises.

\subsection{The condition $Q=6$}
Rather than fixing $A=0$, we can fix the $Q$-curvature. We recall that the $4$-dimensional $Q$-curvature is
\begin{equation}
\label{eq4.16}
\begin{split}
Q:=&\, \frac16 \left [ -\Delta_g \scal_g + \scal_g^2-3|\ric_g |_g^2\right ]\\
=&\, -\frac16 \Delta_g A_g +\frac16 \left ( A_g-12 \right )^2 -\frac12 \left\vert E_g-3g \right\vert_g^2\\
=&\, -\frac16 \Delta_g A_g -A_g -\frac12 |E_g|_g^2+\frac16 A_g^2 +6\ ,
\end{split}
\end{equation}
so Einstein $4$-metrics have $Q=6$. This motivates us to consider replacing the condition $A=0$ by the condition $Q=6$.

Using \eqref{eq3.10} and the last line of \eqref{eq3.9} (and using \eqref{eq2.7} to observe that the $|E|^2$ term is of lower weight), we may rewrite the condition $Q=6$ as
\begin{equation}
\label{eq4.17}
\frac16 x^4 \tr h_x^{(4)} -\frac16 x^3 \tr h_x^{(3)} +\frac23x^2 \tr h_x'' -2x\tr h_x'=\LWT = \frac23 x^2\scal_{h_0} + {\mathcal O}(x^3) \ .
\end{equation}
Differentiating once and setting $x=0$, we immediately see that $\trace_{h_0}h'(0)=0$, which is the same result as we obtained by setting $A=0$. As before, set the free data $\tracefree_{h_0}h'(0)$ to vanish as well, so that $h'(0)=0$. Then we can differentiate \eqref{eq4.17} twice and set $x=0$ to obtain $\trace_{h_0}h''(0)= -\frac12 \scal_{h_0}$, which is the same condition as arises from setting $A=0$, see \eqref{eq4.11}.

Since $\tracefree_{h_0}h''(0)$ is free data for the equation $B=0$ as well as for the equation $Q=6$, let us now choose $\tracefree_{h_0}h''(0)=Z_{h_0}$, so that $h''(0)=-2P_{h_0}$ where $P_{h_0}$ is the Schouten tensor of $h_0$. That is, we choose data that correspond to Poincar\'e-Einstein metrics to order $x^2$ inclusive. One can now compute the coefficients of the higher-order terms in \eqref{eq4.17}. The coefficient of the $x^3$ term vanishes, hence $\trace_{h_0}h'''(0)=0$. (Again, the tracefree part is free data for $B=0$ as well as for $Q=6$.)

To go to order $x^4$ and beyond, differentiate \eqref{eq4.17} $k\ge 4$ times, setting $k=0$, and using the choices and results listed in the last paragraph, we now find
\begin{equation}
\label{eq4.18}
\frac16 k(k-4)\left ( k^2-3k+6\right ) \trace_{h_0}h^{(k)}(0) = \begin{cases} F(h_0),& k=4,\\ F(h_0,h^{(5)}(0),\dots,h^{(k-1)}(0)),& k\ge 5.\end{cases}
\end{equation}
From the left-hand side of \eqref{eq4.18}, we see that $\trace_{h_0}h^{(4)}(0)$ is not determined by the condition $Q=6$. But given $h_0$ and the above choices $\trace_{h_0}h'(0)=0$ and $\tracefree_{h_0}h''(0) =\tracefree_{h_0}P_{h_0}\equiv Z_{h_0}$, if we also choose values for $\tracefree_{h_0}h^{(3)}(0)$ and $\trace_{h_0}h^{(4)}(0)$ then all higher-order traces are determined by recursive application of \eqref{eq4.18}. Now since the left-hand side of \eqref{eq4.18} vanishes when $k=4$, we observe that we must have $F(h_0)=0$ on the right-hand side. But $F(h_0)$ can be separately computed explicitly. We have done so and find that $F(h_0)=0$. This can also be seen without the explicit calculation, by the following argument. The chosen data and the datum $\trace_{h_0}h^{(3)}(0)=0$ together imply that $g$ is asymptotically Poincar\'e-Einstein (APE). Any APE $4$-metric has $|E|\in{\mathcal O}(x^4)$ \cite{BMW1} and $A\in{\mathcal O}(x^5)$ (\cite{Woolgar}, or simply refer to the preceding subsection). Hence these data alone guarantee that $Q-6\in{\mathcal O}(x^5)$ so the fourth-order Taylor coefficient in the expansion of $Q$ vanishes. But this coefficient is $F(h_0)$ (times a non-zero constant).

Thus we have shown the following.

\begin{proposition}\label{proposition4.6}
For $n=4$, choose $\tracefree_{h_0} h'(0)=0$, $\tracefree_{h_0} h''(0)=-2Z_{h_0}$, and $\tracefree_{h_0} h'''(0)=\Psi$ and $\trace_{h_0}h^{(4)}(0)=T_4$ for some arbitrary function $T_4$. If $B_g=0$ and $Q_g-6=0$, then $\trace_{h_0}h^{(k)}$ is uniquely determined for all $k$.
\end{proposition}

\subsection{The main theorems for $n=4$.} We now have assembled everything we need to prove the $n=4$ results quoted in the Introduction.

\begin{proof}[Proof of Theorem \ref{theorem1.1}]
Corollary \ref{corollary4.3} allows us to compute iteratively and uniquely the tracefree parts of $h^{(k)}(0)$ for $k\ge 4$ in terms of a boundary metric $h(0)=h_0$, arbitrary data $h'(0)$, $h''(0)$, $h'''(0)$, and the traces $T_i:=\trace_{h_0}h^{(i)}(0)$ for $4\le i <k$. Proposition \ref{proposition4.4} imposes one restriction on the data, namely that $\divergence_{h_0}\tracefree_{h_0} h^{(3)}(0)=0$.
\end{proof}

\begin{proof}[Proof of Corollary \ref{corollary1.2}]
By \cite{FG1, FG2}, given $h_0$ there is a unique (formal series expansion for a) Poincar\'e-Einstein metric in normal form \eqref{eq1.3} with $h(0)=h_0$, such that $h'(0)=0$, $h''(0)=-2P_{h_0}$, $h'''(0)=\Psi$, and $\trace_{h_0}h^{(4)}(0)=(3!) \left \vert P_{h_0}\right \vert_{h_0}^2$ as well. But every Poincar\'e-Einstein metric is Poincar\'e-Bach, so this formal series represents a Poincar\'e-Bach metric.

Now Proposition \ref{proposition4.5} and Corollary \ref{corollary4.3} allow us to compute iteratively and uniquely both the trace and tracefree parts of $h^{(k)}(0)$ for $k\ge 4$ in terms of a boundary metric $h(0)=h_0$ and arbitrary data $\tracefree_{h_0} h'(0)$, $\tracefree_{h_0} h''(0)$, $\tracefree_{h_0}h'''(0)$, and $T_4:=\trace_{h_0}h^{(4)}(0)$, yielding a unique formal power series for $h_x$ in \eqref{eq1.1}. Choose the data so that $\tracefree_{h_0} h'(0)=0$, $\tracefree_{h_0} h''(0)=-2Z_{h_0}$, $\tracefree_{h_0}h^{(3)}(0)=\Psi$, and $\trace_{h_0}h^{(4)}(0)=(3!) \left \vert P_{h_0}\right \vert_{h_0}^2$. In particular, this series will have $h'(0)=0$, $h''(0)=-2P_{h_0}$, and $h'''(0)=\Psi$.

Thus the free data for the series that solves $B=0$ agrees with the corresponding coefficients of a unique Poincar\'e-Einstein metric. Since the data uniquely determine the full series, and since there exists a formal Poincar\'e-Einstein metric with these data, the formal solution of $B=0$ determined by these data must be Poincar\'e-Einstein.
\end{proof}

\begin{proof}[Proof of Corollary \ref{corollary1.3}]
The proof is the same except that it relies Proposition \ref{proposition4.6} rather than Proposition \ref{proposition4.5}.
\end{proof}

\section{Higher dimensions}
\setcounter{equation}{0}

The Bach tensor as defined by \eqref{eq1.1} is not conformally invariant for $n\ge 5$ (see \cite[equation 4.16]{Bergman}). More importantly for present purposes, the divergence of the tensor defined by \eqref{eq1.1} is not identically zero when $n\ge 5$. To apply the Fefferman-Graham procedure in its usual form, we need a divergence-free generalization.

On a closed manifold, the Euler-Lagrange equation of the action
\begin{equation}
\label{eq5.1}
\begin{split}
S=&\, \int_M \left [ |E|^2-\frac{nA^2}{4(n-1)}+(n-2)A-2(n-1)(n-2)\right ] dV\\
=&\, \int_M \left [ |\ric|^2 -\frac{n\scal^2}{4(n-1)}-\frac12 (n-2)(n-4)\scal -\frac14 (n-1)(n-2)^2(n-4)\right ] dV
\end{split}
\end{equation}
is
\begin{equation}
\label{eq5.2}
0={\hat B}_{ab}:= B_{ab}-\frac{(n-4)}{2(n-2)^2}\left \{ \frac{n}{(n-1)}AE_{ab} -2E_{ac}E_b{}^c +\left [ |E|_g^2-\frac{(n+2)}{4(n-1)}A^2\right ] g_{ab} \right \}\ .
\end{equation}
Since ${\hat B}$ is obtained from an action principle, it is divergence-free (though not tracefree, unless $\left \vert E_g\right \vert_g^2=\frac{1}{4(n-1)} A_g^2$). This can also be checked by explicit calculation.

While the generalization ${\hat B}$ is not unique amongst higher dimensional generalizations of the Bach tensor, in addition to being divergence-free (and thus yielding in direct fashion to the Fefferman-Graham technique), it has the same principal part on a fixed background (in particular, standard hyperbolic $n$-space) as $B$. Amongst actions that are quadratic in $W$, $E$, and $A$, these conditions fix the action \eqref{eq5.1} up to addition of the integral of $c|W|^2$, which would add terms of the form $W_{ijkl}E^{jl}$ to ${\hat B}$. Terms of this form modify our series expansions only at order $x^n$ and beyond, and will not affect our conclusions. In this sense, our results are in fact general.

We present this example because it illustrates that the free data split into low order and high order pairs. The former pair consists of $h(0)$ and $h'(0)$ (for simplicity, we will set $h'(0)=0$), while the latter pair consists of $h^{(n-2)}(0)$ and $h^{(n-1)}(0)$.

\begin{theorem}\label{theorem5.1}
Let $(M,g)$ be asymptotically hyperbolic and either even-dimensional, or odd-dimensional with conformal infinity whose Fefferman-Graham obstruction tensor \cite{FG1} (see also \cite[section 3.1]{DGH}) vanishes. Let $g=\frac{1}{x^2}\left ( dx^2\oplus h_x\right )$ be Poincar\'e-Bach with $-(n-2)\trace_{h_0}h''(0)=\scal_{h_0}\neq 0$ and $h'(0):=h_x'\big \vert_{x=0}=0$. Let $\Phi$ and $\Psi$ be tracefree symmetric $(0,2)$-tensors on conformal infinity such that $\Psi$ is $h_0$-divergenceless. Then for each such $\Phi$ and $\Psi$ there is a unique formal power series solution of the equations ${\hat B}_g=0$ such that $\tracefree_{h_0}h^{(n-2)}(0)=\Phi$, $\tracefree_{h_0}h^{(n-1)}(0)=\Psi$.
\end{theorem}

The condition $h'(0)=0$ is imposed only for tractability and focus. Without it, some of our expressions become quite complicated without compensating gains in insight. The condition on $\scal_{h_0}$ is related to the part of the ${\hat B}=0$ equation that fixes the conformal gauge. This equation is merely quasi-linear, and the Frobenius-type technique used in Fefferman-Graham type analyses can fail. It happens not to fail when this condition is met. The restriction on the obstruction tensor ensures the expansion will be only in powers of $x$; no logarithmic terms will be required. There are no new, further obstructions to formal power series solutions beyond the obstruction in odd bulk dimension already known from the Poincar\'e-Einstein case \cite{FG1, FG2} (at least when $\scal_{h_0}\neq 0$).

\begin{proof}[Sketch of proof] The divergence-free condition for ${\hat B}$ allows us to repeat the analysis of Section 3.2 straightforwardly, and leads to the conclusion that the divergence of $\Phi$ is free data, while the divergence of $\Psi$ is not. However, since ${\hat B}$ is not tracefree, the analysis of ${\hat B}_{00}$ is modified, the effect of which is that the mass aspect is no longer free data.

Now consider the expansion of ${\hat B}^{\perp}$. By inserting equations \eqref{eq2.7} into \eqref{eq3.13} and counting weights, and observing that the difference between $B$ and ${\hat B}$ consists entirely of lower weight terms, we see that the components of ${\hat B}^{\perp}$ are given by
\begin{equation}
\label{eq5.3} {\hat B}_{ij} = -\frac{1}{2(n-2)} x^2 \tf h_x^{(4)} +\frac{(n-4)}{(n-2)}x\tf h_x^{(3)} - \frac{(n-4)(n-3)}{2(n-2)} \tf h_x''+\LWT\ .
\end{equation}
Then, as with the $n=4$ case, we set ${\hat B}_{ij}=0$ on the left of \eqref{eq5.3} and take derivatives with respect to $x$. This yields
\begin{equation}
\label{eq5.4}
(s-n+2)(s-n+1)\tracefree_{h_0}h_{ij}^{(s)}(0)=F_{ij}(h_0,h'(0),\dots,h^{(s-1)}(0))
\end{equation}
for $s\ge 2$, where $F$ denotes a tracefree symmetric $(0,2)$-tensor ($F$ is a generic notation, not meant to denote the same $F$ as elsewhere).

We can then obtain the analogue of Corollary \ref{corollary4.3} for $n\ge 5$, which is tedious to write out but its content is straightforward. It says that the tracefree parts of the coefficients $h^{(k)}(0)$ are either free data or functions of lower order free data. The free data are $h_0$, $h'(0)$, $h^{(n-2)}(0)$, $h^{(n-1)}(0)$, and the traces $\trace_{h_0}h^{(r)}$ for all $r$. For example, for $n=5$ dimensions, we have
\begin{equation}
\label{eq5.5}
\tracefree_{h_0} h^{(s)}(0)=\begin{cases} F(h_0,h'(0)), & s=2\ ,\\
\\
F(h_0,h'(0),\trace_{h_0} h''(0),h^{(3)}(0),h^{(4)}(0)), & s=5\ ,\\ \\
F(h_0,h'(0),\trace_{h_0} h''(0),h^{(3)}(0),h^{(4)}(0),\trace_{h_0} h^{(5)}(0),\\
\qquad \dots, \trace_{h_0} h^{(s-1)}(0)), & s\ge 6\ . \end{cases}
\end{equation}
In any dimension, when $h'(0)=0$ we have $\tracefree_{h_0} h''(0)=-\frac{2}{(n-3)}Z_{h_0}$ where $Z_{h_0}:=\tracefree_{h_0} \ric_{h_0}$.

In any dimension, we must give special consideration to the cases $s=n-1$ and $s=n-2$ (the gaps at $s=3,4$ in our $n=5$ example above). For these cases, the left-hand side of \eqref{eq5.4} vanishes, so the corresponding derivatives cannot be determined. However, it is not clear that the right-hand sides vanish. This is the question of obstructions to formal power series solutions.

To make things more definite, we choose $h_0$ such that $\scal_{h_0}\neq 0$. Since we've chosen $h'(0)=0$, then $\tracefree_{h_0}h''(0) =-\frac{2}{(n-3)}Z_{h_0}$, and we can choose the trace so that $\trace_{h_0}h''(0) =-\frac{1}{(n-2)} \scal_{h_0}$, as required in the statement of the theorem.

Having determined $h''(0)$, we proceed by induction. Say for some $2\le k< n-3$, we have determined $h^{(k)}(0)$. We can then iterate using \eqref{eq5.4} to determine $\tracefree_{h_0}h^{(k+1)}(0)$ and $\trace_{h_0}h^{(k+1)}(0)$. When $s=n-2$, the induction halts, but we are free to choose $\tracefree_{h_0}h^{(n-2)}(0)$. At next order, the divergence-free part of $\tracefree_{h_0}h^{(n-1)}(0)$ is also free data, while the divergence is determined by the lower-order terms already computed. After the free data are fixed, the induction resumes, and all higher order derivatives of $h^{(s)}(0)$, $k\ge n$, are determined iteratively in terms of the free data.

It remains to check that the right-hand side of equation \eqref{eq5.4} vanishes when $s=n-2$ and when $s=n-1$. For $s=n-2$, this follows because up to that order the metric has the same coefficients as a Poincar\'e-Einstein metric, and Poincar\'e-Einstein metrics obey ${\hat B}=0$. The ${\hat B}=0$ metrics can potentially differ from Poincar\'e-Einstein metrics only at next order, where $\Phi$ is a free choice, whereas for Poincar\'e-Einstein metrics it is not.

Finally, we deal with $s=n-1$. In this case, the right-hand side of equation \eqref{eq5.4} must vanish for all $\Phi:=\tracefree_{h_0}h^{(n-2)}(0)$, not just the value given by the corresponding term in a Poincar\'e-Einstein metric. But simple counting shows that the total number of $x$-derivatives (i.e., summed over all occurrences of $x$-derivatives of $h_x$) within any single term on the right-hand side of \eqref{eq5.4} cannot exceed $n-1$, so any $h^{(n-2)}(0)$ appearing in any such term must multiply $h'(0)$ or not multiply any $x$-derivative of $h_x$ at all; e.g., it could possibly multiply $\scal_{h_0}$, say. But since $h'(0)=0$ by assumption, the former possibility is excluded, while the latter possibility is ruled out by parity. (That is, if one expands \eqref{eq3.13} using \eqref{eq2.7} to obtain $B^{\perp}$ as a sum of terms, each composed of $x$-derivatives of $h_x$ multiplying powers of $x$, the sum of the number of $x$-derivatives and the power of $x$ is even.) This is also true of the terms in ${\hat B}^{\perp}$. Therefore, the sum mod $2$ of the number of $x$-derivatives (acting on $h_x$) in each term of \eqref{eq5.4} when $s=n-1$ equals $n-1\mod 2$, ruling out terms of the form $h^{(n-2)}(0)\cdot f(h_0)$ (where by $f(h_0)$ we of course mean any function of $h_0$, its intrinsic connection $D$, etc). Hence the right-hand side $F$ of \eqref{eq5.4} is independent of $\Phi$ and so depends only on $(h_0,h'(0),\trace_{h_0}h''(0))$. But, if $h_0$ yields an unobstructed Poincar\'e-Einstein metric, then $F=F(h_0,0,-\frac{1}{(n-2)} \scal_{h_0})$ must vanish, and does so independently of $\Phi$.
\end{proof}

We remark that the condition in the last line of the proof that $h_0$ should be data for a formal series for a Poincar\'e-Einstein metric (i.e., that the Fefferman-Graham ambient obstruction tensor for $h_0$ vanishes) always holds if the bulk dimension $n$ is even, and holds for odd $n$ if, for example, $h_0$ is conformally Einstein. Also, when $n$ is even, the argument given to rule out obstructions at order $s=n-1$ combined a parity argument with an appeal to the Poincar\'e-Einstein case but this appeal is really just a short-cut. One can use parity alone to complete the argument when the bulk dimension $n$ is even. If $n$ is odd, one can similarly show that there is no order $(n-2)$ obstruction purely by parity considerations, without appeal to the existence of a Poincar\'e-Einstein metric (but of course this would not work at order $(n-1)$ for $n$ odd).

\end{document}